\documentclass{amsart}
\usepackage{verbatim,amssymb,amsmath,amscd,latexsym,amsbsy,mathrsfs} 
\usepackage{graphicx}
\input{xy}
\xyoption{all}
\usepackage{color}

\newtheorem{thm}{Theorem}[section]
\newtheorem{cor}[thm]{Corollary}

\newtheorem{rmk}[thm]{Remark}

\newtheorem{quest}[thm]{Question}

\DeclareMathOperator*{\coker}{coker} 
\DeclareMathOperator*{\rank}{rank} 

\newcommand{\RR}{\mathbb{R}}

\newcommand{\PP}{\mathbb{P}}

\newcommand{\cP}{\mathcal{P}}

\newcommand {\C} {{\mathbb C}}
\newcommand {\R} {{\mathbb R}}
\newcommand {\Z} {{\mathbb Z}}

\newcommand {\OO} {{\mathcal O}}

 \newtheorem{lemma}{Lemma}[section]
 \newtheorem{prop}{Proposition}[section]

\numberwithin{equation}{section}

\begin{document}
\title{ When are braid groups of manifolds K\"ahler?}
\author{
        Donu Arapura    
}
 \thanks {Partially supported by a grant from the Simons Foundation}
\address{Department of Mathematics\\
Purdue University\\
West Lafayette, IN 47907\\
U.S.A.}
 \maketitle

\begin{abstract}
The main result is that, with two trivial exceptions, the pure braid
group  of a Riemann surface  with at least 2  strands is not K\"ahler,
 i.e. it is not the fundamental group of a compact
K\"ahler manifold. This deduced with the help of some homological
properties of these groups established beforehand. 
The braid group of a projective manifold of complex dimension 2 or more is shown to be K\"ahler.
\end{abstract}

Some years ago, the author \cite{arapura} observed that a pure Artin
braid group $P_n$ is not K\"ahler, i.e. it is not the fundamental group of a
compact K\"ahler manifold. This was deduced from a result of Bressler,
Ramachandran and the author \cite{abr} which showed that
 a K\"ahler group cannot be too big. More precisely, a K\"ahler group
 cannot be an extension of a group with infinitely many ends by a
 finitely generated group.
In this follow up, we have
tried to determine whether or not braid groups of some other manifolds
are K\"ahler. We recall that the $n$-strand braid (respectively pure braid) group
 $B_n(X)$ ($P_n(X)$) of a manifold $X$ is the fundamental group of the
configuration space of $n$  distinct unordered (respectively ordered) 
points of $X$.  When $X=\R^2$, these are the usual Artin braid groups. The first  result of this paper,
theorem~\ref{thm:abPn}, studies homological properties of the 
natural homomorphism $J_n:P_n(X)\to \pi_1(X)^n$, when $X$ is a compact
oriented surface with genus $g\ge 1$.
 Part A  of the theorem shows that $J_n$ induces an isomorphism of
 abelianizations  $H_1(P_n(X))\cong H_1(\pi_1(X)^n)$, part B gives a similar statement
with local coefficients when $g\ge 2$, and part C shows the pullback of certain
higher cohomology classes to $P_n(X)$ vanish. The theorem has a number
of purely group theoretic corollaries. Corollary \ref{cor:SigmaPg},
which follows from part B, shows that $P_n(X)$
cannot surject  onto the fundamental group of a
surface of genus larger than $g$, when $g\ge 2$. The theorem is proved by
analyzing the Leray spectral sequence for the inclusion of the ordered
configuration space into $X^n$.

It is known that the pure spherical  braid group $P_n(\PP^1)$ is trivial when $n=2$,
 and nontrivial but finite when $n=3$.
Theorem~\ref{thm:hypRS}, which is our main result,
 says that, with these two exceptions, a pure braid group of  a Riemann surface
 with at least  $2$ strands is  never K\"ahler.  For some cases, such as  when $X$ is
noncompact  and hyperbolic,
the argument is essentially the same as for Artin braid groups. When
$X$ is compact with postive genus $g$, however, the proof is completely different.
In outline when $g\ge 2$ and $n=2$,
suppose that $M$ were a compact K\"ahler manifold with $\pi_1(M)=P_2(X)$.
A theorem of
Beauville-Catanese-Siu would show that $J_2$ is realized by a holomorphic map $f$
of $M$ to  a product $Y_1\times Y_2$ of  a pair of Riemann surfaces of genus $g$. Part A
of theorem~\ref{thm:abPn} tells us that $f^*H^1(Y_1\times Y_2)$ equals
$H^1(M)$. Part C when combined with corollary  \ref{cor:SigmaPg} would imply that
$Y_1=Y_2$ and that $f$ would factor
through  the diagonal. But this would force 
$f^*H^1(Y)$  to be strictly  smaller than $H^1(M)$, causing  a contradiction.

In the positive direction, we show that when $X$ is a projective
manifold of complex dimension 2 or more, $B_n(X)$ is the fundamental group of a
projective manifold, for any $n$. In this case, $B_n(X)$ has
a very simple structure: it is the wreath
product $\pi_1(X)\wr S_n$, or in other words, it is the semidirect
product $\pi_1(X)^n\ltimes S_n$. As a consequence, we see that the class of
fundamental groups of projective manifolds is closed under taking
these kinds of wreath products. This includes an old result of Serre
that all finite groups arise as fundamental groups of projective
manifolds. Remark \ref{rmk:Serre} gives  a bit more explanation of
how this relates to previous arguments. Although this paper is mainly
about group theory and complex geometry, the last section discusses
some (potential) analogues in positive characteristic.

Manifolds, in this paper, are  assumed to be connected, unless stated otherwise.
Base points will be generally omitted. 
Since we will switch between real and complex manifolds, we will try to be
clear on what kind of dimension we mean.

\section{Some homological properties of braid groups}

Given a  manifold (or sufficiently nice space) $ X$ let 
$$\Delta_{n,ij}(X) =
\{(x_1,\ldots, x_n)\in X^n\mid 
x_i=x_j\},$$
 and $\Delta_n(X)= \bigcup \Delta_{n,ij}(X)$. We will write
$\Delta_{ij}=\Delta_{n,ij}(X)$ and $\Delta=\Delta_{n}(X)$ if $n$ and $X$ are understood.
 Let $C_n(X)=  X^n-\Delta_n(X)$ and $SC_n(X)=C_n(X)/S_n$
denote the associated  configuration spaces. The  braid group on
$n$-strands for $X$ is defined by $B_n(X)= \pi_1(SC_n(X))$, and the pure braid group by
$P_n(X)= \pi_1(C_n(X))$. We have the standard exact sequence
$$1\to P_n(X)\to B_n(X)\to S_n\to 1$$
so the two groups are closely related.  Fadell and Neuwirth \cite{fn} showed that the
various projections $C_n(X)\to C_m(X)$ are fibrations. Let us assume
that the base point (which will be suppressed
in the notation) is  $(x_1,\ldots, x_n)$. Then we get an exact
sequence
\begin{equation}
  \label{eq:FN}
\ldots \pi_2(C_{m}(X))\to  P_{n-m}(X-\{x_{1},\ldots, x_m\})\to P_n(X)\to P_{m}(X)\to 1  
\end{equation}

An easy
induction  with this sequence shows that:

\begin{lemma}\label{lemma:FN}
If for any finite set $S$ and $1\le i\le N$, $\pi_i(X-S)=0$,
then $\pi_i(C_n(X))=0$ for $1\le i\le N$. In particular, \eqref{eq:FN}
gives a short exact sequence of fundamental groups.
\end{lemma}

 The inclusion $j:C_n(X)\to X^n$ induces a homomorphism
 $J_n:P_n(X)\to \pi_1(X^n)=\pi_1(X)^n$. When $X$ is a manifold of (real) dimension at least
 $2$, Birman \cite[thm 1.5]{birman} showed that $J_n$ is surjective. 
 
  In the following, we let
 $X_g$ denote {\em the} compact oriented real $2$-manifold of genus
 $g$. Let $P_n(g) = P_n(X_g)$, $B_n(g)=B_n(X_g)$ and
 $\Pi_g=\pi_1(X_g)$. We have the standard presentation
$$\Pi_g = \langle \alpha_1,\ldots, \alpha_{2g}\mid
[\alpha_1,\alpha_2]\ldots [\alpha_{2g-1},\alpha_{2g}]=1\rangle$$
We will usually assume $g>0$ below, so that $X_g^n$ is a
$K(\Pi_g^n,1)$, and by the above lemma, $C_n(X_g)$
is a $K(P_n(g),1)$. 
 Consequently, the (co)homology of $P_n(g)$ and $C_n(X_g)$
 with local coefficients are the same. 
Recall that a local system on a connected space $T$ can be viewed
as $\pi_1(T)$-module or a locally constant sheaf on $T$.

\begin{thm}\label{thm:abPn}
 Suppose that $n\ge 2$ and $g\ge 0$ are  integers.
 \begin{enumerate}
 \item[(A)] If $g>0$, $J_n$ induces an isomorphism of abelianizations
 $$H_1(P_n(g))\cong H_1(\Pi_g)^n$$
Otherwise, if $g=0$, $H_1(P_n(0))$ is free of rank $\binom{n}{2}-n$ if
$n\ge 3$.

\item[(B)] If $g\ge 2$, then for any irreducible $\C[\Pi_g^n]$-module $V$,
  $J_n$ induces an isomorphism
$$ H^1(\Pi_g^n, V)\cong H^1(P_n(g), V)$$
When  $g=1$ and $V$ is a one dimensional $\C[\Pi_1^n]$-module, 
$\dim H^1(P_n(1),V)$ is given by the combinatorial formula \eqref{eq:dimH1Pi1} explained below.

\item[(C)] 
Let $g>0$ and $2\le m\le n$ and let $\pi:P_n(g)\to  \Pi_g^m$ be given
by composing $J_n$ with some projection onto $m$ factors.
Then $\pi^*:H^{2m}(\Pi_g^m)\to H^{2m}(P_n(g))$ is zero.

 \end{enumerate}
 
\end{thm}

Part A is a bit surprising at first glance. The kernel of $J_n$ is
quite large. It contains loops around the diagonals $\Delta_{ij}$.
These classes will span the abelianization of the spherical braid
group $P_n(0)$, but  the theorem says that they will not contribute to
$H_1(P_n(g))$ for positive $g$.  Part B will have some group theoretic
consequences given in the corollaries. We now explain the missing formula for $\dim H^1(P_n(1),V)$.
Let
$$Char(\Pi_1^n) := Hom(\Pi_1^n, \C^*) =\prod_1^n Char(\Pi_1,\C^*)\cong \prod_1^n (\C^*)^2$$
Given a character $\rho\in Char(\Pi_1^n) $,
let $\C_\rho$ denote the corresponding $\C[\Pi_1^n]$-module.
We can decompose  $\rho=(\rho_1,\rho_2,\ldots)$ as above, where $\rho_i\in Char(\Pi_1,\C^*)$.
Then
\begin{equation}
  \label{eq:dimH1Pi1}
\dim H^1(P_n(1), \C_\rho) = \#\{(i,j)\mid 0\le i< j\le n, \rho_i\rho_j=1\}  
\end{equation}

\begin{proof}[Proof of theorem \ref{thm:abPn}]
Set $X=X_g$.
We  reformulate part A cohomologically. Recall, that  if $T$ is a space with
finitely generated first homology,  the universal coefficient theorem
gives isomorphisms
$$H^1(T)\cong Hom(H_1(T), \Z)$$
$$H^2(T)_{torsion} \cong H_1(T)_{torsion}$$
Here we use integral coefficients.
The first isomorphism is canonical, but  the second is not. In view of this, it
suffices to prove that
\begin{equation}
  \label{eq:H1Pn}
H^1(P_n(X))\cong H^1(X)^n  
\end{equation}
and that
$H^2(P_n(X))$
is torsion free.  

We first give the proof of A and B  when $n=2$. This  will explain the
main ideas without a lot of homological baggage.
  We have a Gysin sequence
$$\ldots H^{i-2}(\Delta)\to H^i(X^2)\to H^i(X^2-\Delta)\to
H^{i-1}(\Delta)\ldots $$
When $i=1$, we can write this as
$$0\to H^1(X^2)\to H^1(X^2-\Delta) \to \Z  \stackrel{d}{\to}  H^2(X^2)$$
where $1\in \Z$ is sent to the fundamental class $[\Delta]$
under $d$. Since this class is nonzero, $d$ is injective. Therefore
\eqref{eq:H1Pn} holds. The Gysin sequence also shows that any possible
torsion in $H^2(X^2-\Delta)$ must come from $\coker d$. Since the
intersection number $\Delta\cdot (X\times {x_1})=1$, $[\Delta]$ is
primitive, i.e. not  divisible by  $m>1$. Therefore
$\coker d$ is torsion free. This proves A.

For B, we can assume that $V$ is nontrivial. As is well known \cite[prop 2.3.23]{kowalski}, we can decompose
$V=V_1\boxtimes V_2=\pi_1^*V_1\otimes \pi_2^*V_2$, where $V_i$ are irreducible
$\C[\Pi_g]$-modules.
One again has a Gysin sequence
$$0\to H^1(X^2, V)\to H^1(X^2-\Delta, V)\to H^0(\Delta,
V_1\otimes V_2)\stackrel{\gamma}{\to} H^2(X^2, V)$$
If $V_1\otimes V_2 \not=1$, then $H^0(\Delta, V_1\otimes V_2)=0$ by Schur's lemma, so
we are done in this case. Now suppose that $V_2\cong V_1^{\vee}$ and nontrivial.
We split into  subcases depending on $g$. First, suppose that $g=1$,
and $V$ is nontrivial one dimensional.
 Then $H^i(X^2,V)=0$.
Therefore 
$$ H^1(X^2-\Delta, V)\stackrel{\sim}{\to}H^0(\Delta, \C)\cong \C$$
Now assume that $g\ge 2$.
The Gysin map $\gamma$ is Poincar\'e dual to restriction as in the  usual
setting. The component 
$$\gamma_{11}: \C=H^0(\Delta, V_1\otimes V_2)\to
H^1(X, V_1)\otimes H^1(X, V_1^\vee)$$
is dual to the cup product pairing
$$H^1(X, V_1)\otimes H^1(X, V_1^\vee)\to H^2(X, \C)=\C$$
which is perfect by Poincar\'e duality.  It is also nonzero, because
\begin{equation}
  \label{eq:chiXg}
  \begin{split}
  -\dim H^1(X, V_1) &=\dim H^0(X, V_1)-\dim H^1(X, V_1) +\dim
H^2(X,V_1)\\
& = 2\dim V_1(1-g) <0
\end{split}
\end{equation}
This proves B.

When $n>2$, the argument is basically the same, except that we use
the Leray spectral sequence for sheaf cohomology
$$E_2^{pq} = H^p(X^n, R^qj_*\Z)\Rightarrow H^{p+q}(C_n(X))$$
Totaro \cite[pp 1060-1062]{totaro} gave a more down to earth
description of the initial term, and the first nontrivial differential,
which is $d_2$ in our case. Here we just describe the part that we need.
We have isomorphisms
\begin{equation}
  \label{eq:leray}
  \begin{split}
E_\infty^{10}=E_2^{10}&= H^1(X^n)\\
  E_2^{01} &= \bigoplus_{1\le i< j\le  n} \Z G_{ij}\\
E_2^{20} &= H^2(X^n)
\end{split} 
\end{equation}
where $G_{ij}$ are basis vectors. The differential $d_2:E_2^{01}\to E_2^{20}$ sends
$G_{ij}\mapsto [\Delta_{ij}]$.
Consider the  K\"unneth decomposition 
\begin{equation}
  \label{eq:kunneth}
H^2(X^n)= \bigoplus_i \pi_i^*H^2(X) \oplus \bigoplus_{i<j}
\pi_i^*H^1(X)\otimes \pi_j^*H^1(X)
\end{equation}
When $g=0$, the sum reduces to a sum of $n$ copies of $H^2(X)\cong
\Z$. Let $e_i$ denote the positive generator on the $i$th copy. Under Gysin,
$[\Delta_{ij}]$ maps to $e_i+e_j$. Therefore  $d_2:E_2^{01}\to
E_2^{20}$  can be identified with a homomorphism
$$\Z^{\binom{n}{2}}\to \Z^n$$
which can be seen to be surjective since we assumed that $n\ge 3$.
So \eqref{eq:H1Pn} is verified in this case. Now suppose that $g>0$.
Then under Gysin,
 $[\Delta_{ij}]$ has nonzero image in the $ij$th summand of the second
 sum of \eqref{eq:kunneth}.
This implies that $d_{2}$ is injective. Therefore
\eqref{eq:H1Pn} holds in the second case. The remainder of the proof of A,
treats $g=0$ and $g>0$ simultaneously. The associated graded of
$H^2(C_n(X))$ is a sum of
$$E_\infty^{02} \subseteq\bigoplus_{i<j<k}  \Z^b$$
$$E_\infty^{11}\subseteq  \bigoplus_{i<j} H^1(X)$$
$$ E_\infty^{20}=\coker d_2: E_2^{01}\to E_2^{20}$$
where the exponent $b$ is the second Betti number of $C_3(\R^2)$ 
(which is complicated way of saying $b=3$). The key point for us is that the  first two terms $E_\infty^{02}, E_\infty^{11}$
are manifestly torsion free. 
The third term $ E_\infty^{20}$ is also
torsion free by an argument similar to the one used in the second paragraph.
Since all of the associated graded groups are torsion free, it follows
that $H^2(C_n(X))$ is torsion free. So A is proved.

For B, we can decompose $V= V_1\boxtimes V_2,\ldots$ as above. We can
assume that $V$ is nontrivial.
We use the spectral sequence
$$E_2^{pq} = H^p(X^n, R^qj_*j^*V) \Rightarrow H^{p+q}(
C_n(X),V)$$
The relevant terms can be simplified, as  Totaro did, to 
$$E_\infty^{10}=E_2^{10}= H^1(X^n,V)$$
$$E_2^{01}= \bigoplus_{1\le i< j\le  n} H^0(\Delta_{ij}, V_i\otimes V_j)$$
$$E_2^{20} = H^2(X^n, V) $$
When $g=1$, $H^i(X,V)=0$, and
$$H^0(\Delta_{ij}, V_i\otimes V_j)=
\begin{cases}
  \C & \text{if } V_i\cong V_j^\vee\\
0 & \text{otherwise}
\end{cases}
$$
This implies that
\begin{equation*}
\dim H^1(P_n(1), V) = \#\{(i,j)\mid 0\le i< j\le n, V_i\cong V_j^\vee\}  
\end{equation*}
which is clearly equivalent to  \eqref{eq:dimH1Pi1}.
So we are done. Now suppose $g\ge 2$.
The differential $d_2:E_2^{01}\to E_2^{20}$  is a sum of Gysin maps in
the sense explained above.
Arguing exactly as we did for the $n=2$ case,  we can see that
either $V_i\ncong V_j^\vee$, in which case $H^0(\Delta_{ij},
V_i\otimes V_j)=0$, or else $V_i\cong V_j^\vee $ and
the portion of the Gysin map given by
$$\C=H^0(\Delta_{ij}, V_i\otimes V_i^\vee)\to H^1(X, V_i)\otimes H^1(X, V_i^\vee)$$
is injective. Therefore
$$H^1(X^n, V)\cong H^1(C_n(X), V)$$
\medskip

It remains to prove part C. Factor $\pi$ through $J_m$, so that $\pi^*$ factors through $H^{2m}(C_m(X))$.
When $m\ge 2$, $C_m(X)$ is a noncompact oriented manifold of real
dimension $2m$,  so this group is zero by Poincar\'e duality.

\end{proof}


For the remainder of this section, we explain some group theoretic
consequences of the last theorem. Some of them will be needed later.

\begin{cor}
  When $g\ge 1$, $H_1(B_n(g))\cong H_1(\Pi_1(g))$. When $X$ is a Riemann
  surface of genus $g$, the Abel-Jacobi map (with respect to any base
  point) induces an isomorphism
$$H_1(C_n(X))\cong H_1(S^nX)\cong H_1(J(X))$$
\end{cor}

\begin{proof}
This follows by taking the $S_n$-coinvariant part of the isomorphism
$$H_1(C_n(X_g))\cong H_1(X_g^n)$$  
The second is immediate using standard properties of Jacobians and
Abel-Jacobi maps \cite{gh}.
\end{proof}

The Burau and Gassner representations \cite{birman} of the pure Artin braid group can be viewed
as a family of complex representations, depending on parameters $t,t_1,\dots$, which specialize to the
 trivial representation. Equivalently, these representations are
 deformations of the trivial representation.
One might hope likewise to construct interesting representations of the pure
braid group in higher genus by deforming an uninteresting one.
``Interesting" should be taken to mean that it is not merely pulled
back from $\Pi_g^n$ under $J_n$. The
next two results will imply that interesting
semisimple representations cannot be produced this
way (see remark \ref{rmk:def}).

\begin{cor}\label{cor:def}
Suppose that  $g>1$. If $V$ is an irreducible $\C[P_g(n)]$-module pulled back
from $\Pi_g^n$, then any small deformation of $V$ is again pulled back from $\Pi_g^n$.
\end{cor}

Before giving the proof, we need formulate this a bit more precisely.
We recall a few facts about representation varieties \cite{lm}.
When  $G$ is a finitely generated group,  the functor from commutative
$\C$-algebras to sets
$$R\mapsto Hom_{groups}(G, GL_r(R))$$
is representable by an affine $\C$-scheme of finite type, whose $\C$-points are $Hom(G,
GL_r(\C))$ (and we use the same notation for the scheme).
The $\C$-points of the GIT quotient
$$Char(G,GL_r):=Hom(G,GL_r(\C))//GL_r(\C)$$
are  equivalence classes of representations,
where two representations are equivalent if they have isomorphic semisimplifications.
Consequently, the points can also be viewed
as isomorphism classes of $r$-dimensional
semisimple representations of $G$. 
There is a possibly empty open set $Char^s(G,GL_r)\subset Char(G,GL_r)$
parametrizing irreducible representations. 
Given a $\C$-point $V\in Char^s(G,GL_r)$,
$H^1(G, End(V))$ is the Zariski tangent space to the scheme at
$V$ \cite[2.4, 2.13]{lm}. The schemes $Hom(G, GL_r(\C))$ and $Char(G, GL_r)$ often go by the
representation and character ``variety''.
The dimension formulas become a bit simpler if we replace $GL_r(\C)$ by $SL_r(\C)$, and define
$$Char(G,SL_r)=Hom(G,SL_r(\C))//SL_r(\C)\subset Char(S, GL_r)$$
In this case, the tangent space at an irreducible representation  is $H^1(G, sl(V))$, where
$sl(V)$ is kernel of the trace $End(V)\to \C$. Here is the
precise form of the previous corollary.

\begin{cor}
Let $g>1$.
If $V\in Char^s(\Pi_g^n, SL_r)$, the map $J_n^*$ induces an
isomorphism between a suitable
 analytic neighbourhood of $V\in  Char^s(\Pi_g^n, SL_r )$ and its pullback
 in $Char^s(P_n(g), SL_r)$. The analogous statement holds for the 
 character variety $Char^s(-,GL_r)$.
\end{cor}

\begin{proof}
Part B of the theorem obviously extends to semisimple representations, and this
implies that $J_n$ induces an isomorphism of
tangent spaces of $Char(\Pi_g^n,SL_r )$ and $Char(P_n(g), SL_r)$ at $V$. If we
can prove that that both schemes are smooth at $V$, then we are done
with the first part
by the implicit function theorem. 

  We first  check that  $S=Char(\Pi_g^n,SL_r)$ is smooth at $V$. When $n=1$,
  this is well known. We have that 
  \begin{equation}
    \label{eq:dimS}
\dim S\le \dim H^1(\Pi_g^n, sl(V))    
  \end{equation}
  where $\dim S$ etc. should be understood as the dimension at $V$, i.e. $\dim \OO_{S,V}$.
For smoothness, it is enough to check the opposite inequality. 
 We can decompose
 $V= V_1\boxtimes V_2\ldots$, where $V_i$ is irreducible of rank
 $r_i$. Therefore the image of the morphism
$$Z=Char(\Pi_g, SL_{r_1})\times Char(\Pi_g, SL_{r_2})\times\ldots\to 
S$$
defined by external tensor product contains $V$. This morphism is injective in
a neighbourhood of $V$. Therefore $\dim S \ge \dim Z=\sum \dim
Char(\Pi_g, SL_{r_i})$. To compute these dimensions, we use the morphism
$$SL_{r_i}(\C)^{2g}\to SL_{r_i}(\C)$$
Given by sending $(A_i)\mapsto [A_1,A_2]\ldots [A_{2g-1},
A_{2g}]$. This can be seen to be surjective with finite fibres.
Therefore 
$$\dim Hom(G, SL_{r_i}(R)) = (r_i^2-1)(2g-1)$$
  and one deduces that
$$\dim Char(\Pi_g, SL_{r_i})= 2(r_i^2-1)(g-1)$$
On the other hand, by K\"unneth and a calculation similar to \eqref{eq:chiXg}
$$\dim  H^1(\Pi_g^n, sl(V))= \sum_i \dim H^1(\Pi_g, sl(V_i)) = \sum_i
2(r_i^2-1) (g-1)$$
This implies that equality holds in \eqref{eq:dimS} as we claimed. 

The pullback map $J_n^*:Char(\Pi_g^n, SL_r)\to Char(P_n(g),SL_r)$ can be
seen to be injective.  Therefore
\begin{equation*}
  \begin{split}
\dim Char(\Pi_g^n, SL_r)&\le \dim  Char(P_n(g),SL_r)\\
 &\le  \dim H^1(P_n(g), sl(V))   \\
&=\dim H^1(\Pi_g^n , sl(V))  
  \end{split}
\end{equation*}
So we must have  equality. This implies that $Char(P_n(g),SL_r)$ is also
smooth $V$.
 
We omit the details for $Char(-, GL_r)$. The proof is almost the same. The only change is 
that the dimensions of $Char(\Pi_g^n,GL_r)$, $Char(P_n(g),GL_r)$,
and their tangent spaces at $V$  are
$$2gn+\sum_i 2(r_i^2-1) (g-1)$$
\end{proof}

\begin{rmk}\label{rmk:def}
The last  result does not preclude the possibility of deforming a semisimple
representation of $\Pi_g^n$ to a nonsemisimple representation of
$P_n(g)$, which is not a $\Pi_g^n$-module.  However, it will imply
that the semisimplification of the new
representation would still be a $\Pi_g^n$-module.
\end{rmk}

For any finitely generated group $G$, let 
$$Char(G) = Hom(G, \C^*)= Hom(H_1(G), \C^*)$$
denote the group of characters of $G$. This is an algebraic torus, i.e. product of $\C^*$'s,
times a finite abelian group.
Given $\rho\in Char(G)$, let
$\C_\rho$ denote the corresponding $\C[G]$-module. 
The first cohomology jump locus
is defined by
$$\Sigma^1(G) = \{\rho\in Char(G)\mid H^1(G, \C_\rho)\not=0\}$$
This invariant is popular among people working with K\"ahler and
related groups. We have the following well known properties

\begin{lemma}
  \-
  \begin{enumerate}
  \item The set $\Sigma^1(G)$ is closed in the Zariski topology.
  \item When $G=\Pi_g$, $g\ge 2$ and $\rho\not= 1$, $\dim
    H^1(G,\C_\rho) =2g-2$. Therefore $\Sigma^1(G)=Char(G)$.
\item If $G= G_1\times G_2\ldots$,  with projections $\pi_i$, $\Sigma^1(G) = \pi_1^*\Sigma^1(G_1)\cup
  \pi_2^* \Sigma^1(G_2)\ldots$
\item If $G \cong \Z^n$, $\Sigma^1(G)=\{1\}$.
\item A surjective homomorphism $f:G_1\to G_2$ gives an inclusion
  $f^*\Sigma^1(G_2)\subset \Sigma^1(G_1)$.
  \end{enumerate}
\end{lemma}

\begin{proof}
 For (1), see \cite[I cor 2.5]{arapura2}.
  The second item follows  from \eqref{eq:chiXg}, the third from the
  K\"unneth formula, the fourth from (3) and  Poincar\'e duality
  $H^1(\Z,\C_\rho)\cong H^0(\Z,\C_{\rho^{-1}})^*$,
 and the fifth from the fact that
  $f^*:H^1(G_2,\C_\rho)\to H^1(G_1,\C_{f^*\rho})$ is injective e.g. by
  Hochschild-Serre.
\end{proof}

For an arbitrary  group, $\Sigma^1(G)$ can be fairly wild.  When $G$ is the fundamental group of a quasiprojective
manifold, this set is always a finite union of translates of subtori \cite{arapura2}.  For braid groups, we can give 
much more precise information. For ease of
reading, we split cases involving $g$ into separate corollaries.

\begin{cor}\label{cor:SigmaPg}
  Suppose $g\ge 2$.
  \begin{enumerate}
  \item  We can identify
$$Char(P_n(g)) = \prod_1^n \pi_i^*Char(\Pi_g)$$
\item The cohomology jump locus $P_n(g)$ is a finite union of  $2g$ dimensional
  subtori, specfically
$$\Sigma^1(P_n(g)) =  \bigcup \pi_i^*Char(\Pi_g)$$
\item There is no surjective homomorphism $P_n(g)\to \Pi_h$
  unless $h\le g$.  
\end{enumerate}
\end{cor}

\begin{proof}
  The first item follows from part A of the theorem, and the second
  from part B and the previous lemma. If we had a surjection
  $f:P_n(g)\to \pi_1(\Pi_h)$, then $f^*Char(\Pi_h)\subset \Sigma^1(P_n(g))$ is
  an $2h$ dimensional subtorus. If $h>g$, then this  contradicts
   (2).
\end{proof}

\begin{cor}\label{cor:SigmaP1}
\-
   \begin{enumerate}
  \item  We can identify
$$Char(P_n(1)) = \prod_1^n \pi_i^*Char(\Pi_1)$$
\item  When $n\ge 2$, $\Sigma^1(P_n(1))$ is a finite union of $2n-2$ dimensional subtori,
$$T_{ij} = \{(\rho_1,\rho_2,\ldots)\mid \rho_i\rho_j=1\}$$
\item For a general point $\rho\in T_{ij}$,
$$\dim H^1(P_n(1),\C_\rho)=1$$

\item When $n\ge 3$, there is no surjective homomorphism $P_n(g)\to
  \Pi_h$, unless $h< n-1$.
\end{enumerate}
\end{cor}

\begin{proof}
  The first item follows from part A of the theorem, (2) and (3) from
  \eqref{eq:dimH1Pi1}. Given a surjection $f:P_n(g)\to \Pi_h$,  $h>n$
  would contradict (1), and $h=n$ would contradict (2). Finally,  suppose $n\ge 3$.
  If $h=n-1$, then $f^*Char(\Pi_h)$ would have to coincide with some $T_{ij}$.
This would force $\dim H^1(P_n(1),\C_\rho)\ge 2(n-1)-2>1$ for every
$\rho\in T_{ij}$, and this would  contradict (3).

\end{proof}

The kernel  $K=\ker J_n:P_n(g)\to \Pi_g^n$ can be described
as the normal subgroup generated by an
embedding of the  Artin braid group  $P_n(\R^2)\subset P_n(g)$ \cite[thm 1.7]{birman}.
At least for $g=1$, the gap between $P_n(\R^2)$ and $K$ is considerable.

\begin{cor}
  If $g=1$ and $n\ge 2$,  $\dim H_1(K,\C)=\infty$. In particular, $K$ is not finitely generated,
  and therefore it does not coincide with $P_n(\R^2)$.
\end{cor}

\begin{proof}
This stems from the fact that $\Sigma^1(P_n(1)) $ is an infinite set
by the last corollary.
  Since the argument is sort of standard (see the proof of \cite[V
  1.10]{arapura2}), we just sketch it. Note that the ``Alexander
  module''  $M=H_1(K,\C)$ is a $\C[A]$-module, where $A=\Pi_1^n$,
via the extension
$$ 1\to K\to P_n(1)\to A\to 1$$
Use Hochschild-Serre 
to obtain
$$0\to H^1(A,\C_\rho)\to H^1(P_n(1), \C_\rho)\to H^0(A,
H^1(K,\C_\rho))\to  H^2(A,\C_\rho)$$
Using the fact that $A$ is abelian, we get
$$H^1(\Pi_1, \C_\rho)\cong H^0(A,
H^1(K,\C_\rho))\cong Hom_{\C[A]}(M, \C_\rho)$$
for $\rho\not=1$. If $M$ was finite dimensional,  this would be zero
for all but a finite number of $\rho$'s.

\end{proof}

We can extend some of  the above  results to  $X=\C^*$.

\begin{prop}\label{prop:Cstar}
 Suppose that $X=\C^*$. Then
 $$\rank H_1(P_n(X)) = n +\binom{n}{2}$$
 If $\rho\in Char(\pi_1(X)^n)=(\C^*)^n$, then
 $$\dim H^1(P_n(X),\C_\rho) = \#\{(i,j)\mid 0\le i< j\le n, \rho_i\rho_j=1\}  $$
\end{prop}

\begin{proof}
   The proof is very similar to  what was done above. However, there is
one new step that we will explain, starting with $n=2$.
Consider the Gysin sequence
$$0\to  H^1(X^2)\to H^1(X^2-\Delta)\to H^0(\Delta)\stackrel{\gamma}{\to} H^2(X^2)$$
 We claim that $\gamma$ is zero.  
There are at least two ways to see
 this. One can see from Deligne \cite{deligne}, that the Gysin map is
 a morphism of mixed Hodge structures,
$$H^0(\Delta)(-1)\to H^2(\C^*\times \C^*)$$
One can check that the Hodge structure on the left is $\Z(-1)$ and on
the right, it is $\Z(-2)$. This forces $\gamma=0$.
For a more pedestrian proof, consider the diagram
$$
\xymatrix{
 H^0(\Delta)\ar[r]^{\gamma} & H^2(\C^*\times \C^*) \\ 
 H^0(\bar \Delta)\ar[u]^{\cong}\ar[r] & H^2(\PP^1\times \PP^1)\ar[u]^{r}
}
$$
where $\bar\Delta\subset \PP^1\times \PP^1$ is the diagonal. Using the K\"unneth
formula, and the fact that $H^1(\PP^1)=H^2(\C^*)=0$, we can see that $r=0$. Therefore, it follows again that
$\gamma=0$. 

With the claim in hand, one can see that
$$H^1(P_2(X)) = H^1(X^2-\Delta)\cong \Z^3$$
For the general case, one checks that the differential
$$d_2:E_2^{01} \to E_2^{20}$$
in \eqref{eq:leray} vanishes, by arguing as in the claim. This will
show that
$$\rank H_1(P_n(X)) = \rank E_2^{10}+\rank E_2^{01}= n +\binom{n}{2}$$

The proof of last part is identical to the proof of part B of
theorem~\ref{thm:abPn}  when $g=1$.

\end{proof}

\begin{cor}\label{cor:Cstar}
Suppose that $n\ge 2$.
\begin{enumerate}
\item The set $\Sigma^1(P_n(X))\cap J_n^*Char(\pi_1(X)^n)$ is a union of tori
 $$T_{ij} = \{(\rho_1,\rho_2,\ldots)\mid \rho_i\rho_j=1\}$$
 
 \item If $\rho\in T_{ij}$ is a general point
 $$\dim H^1(P_n(X),\C_\rho)=1$$
 
 \item  There is no surjective homomorphism $f:P_n(X)\to \Pi_g$, with $g\ge 2$, such that $f^*Char(\Pi_g)$ contains $T_{ij}$ 
\end{enumerate}
\end{cor}

\section{Braid groups of Riemann surfaces are almost never K\"ahler}

A group $G$ is called K\"ahler if it is isomorphic to the fundamental
group of a compact K\"ahler manifold. Let us say $G$ is projective
if it is isomorphic to the fundamental
group of a projective manifold. Projective groups are K\"ahler, but
the converse is unknown.
Here is the main result of the paper.

\begin{thm}\label{thm:hypRS}
If $X$ is a Riemann surface and $n\ge 2 $, then  $P_n(X)$ is not K\"ahler,
unless $X=\PP^1(= S^2)$ and $n=2$ or $3$.
\end{thm}

\begin{rmk}
The exceptions $P_2(\PP^1)$ and $ P_3(\PP^1)$ are either trivial or finite \cite[p 34]{birman}. 
\end{rmk}

\begin{cor}
  With the same assumptions as above, the group $B_n(X)$ is not K\"ahler.
\end{cor}

\begin{proof}
 This is a consequence of the fact that a subgroup of finite index in
 a K\"ahler group is also  K\"ahler.
\end{proof}

The rest of the  section will be devoted to the proof of the theorem.
We split the proof into several cases.  The assumption that $n\ge 2$ will be in force for the rest of this section.

\begin{lemma}\label{lemma:artin}
  The theorem holds when $X$ is $\C$ or the disk.
\end{lemma}

\begin{proof}
  In either case, $P_n(X)$ is a pure Artin braid group, and the result was
  proved in \cite[\S 3]{arapura}. The main point is that $P_n(X)$ can
  be written as an extension of group with infinitely many ends by a
  finitely generated group. Such a group cannot be K\"ahler by
a theorem of Bressler, Ramachandran and the author \cite[cor
4.3]{abr}.
We recall that a group has 
infinitely many ends if a Cayley graph for it does. 
\end{proof}

\begin{lemma}
  Theorem \ref{thm:hypRS} holds when $X$ is noncompact  nonsimply connected hyperbolic Riemann surface.
\end{lemma}

\begin{proof}
This uses the same strategy as above.
The assumptions imply that $\pi_1(X)=F$ is nonabelian free. Such a
group has infinitely many ends.
We can assume that $F$ is finitely generated, since it would not be
K\"ahler otherwise.
 By \eqref{eq:FN}, we have an exact sequence
$$1\to P_{n-1}(X-\{x_0\})\to P_{n}(X)\to F\to 1$$
This means that $P_n(X)$ is an extension of a  group with infinitely many ends
by  a finitely generated group. Such a group cannot be K\"ahler as
noted above. 
\end{proof}

\begin{lemma}\label{lemma:P1}
 Theorem \ref{thm:hypRS} holds when $X=\PP^1$ and $n\ge 4$.
\end{lemma}

\begin{proof}
 We note that $P_3(\PP^1)=\Z/2\Z$ and  $\pi_2(C_3(\PP^1))=0$ \cite[p 34]{birman}.
  Therefore, by \eqref{eq:FN}, we have an exact sequence
$$1\to \pi_1(\PP^1-\{x_1,x_2,x_3\})\to P_{4}(\PP^1)\to P_3(\PP^1)\to 1$$
This means that $P_4(\PP^1)$ contains a nonabelian free subgroup of
finite index. Applying \eqref{eq:FN} again
shows that, when $n\ge 4$, $P_n(\PP^1)$
 contains a subgroup of finite index which is an extension of a group with 
infinitely many ends by a finitely generated group. Therefore $P_n(\PP^1)$ cannot be K\"ahler.
\end{proof}

It remains to treat the case when $X$ is compact with positive genus or $\C^*$. This 
requires a completely new strategy. We start with a result
needed to justify one step of the proof.

\begin{lemma}\label{lemma:inj}
  Suppose that $f:M\to N$ is a proper surjective holomorphic map of 
 complex manifolds, with $M$ K\"ahler.  Then
  $f^*:H^{i}(N, \R)\to H^{i}(M,\R)$ is injective for all $i$.
\end{lemma}

\begin{proof}
Let $d$ be the complex dimension of $N$, and let $e=\dim_\C M- \dim_\C N$. Fix a K\"ahler metric
on $M$ with K\"ahler form $\omega$. Suppose that $\alpha$ is a  closed $C^\infty$ $i$-form defining a nonzero class in
  $H^{i}(N,\R)$. By Poincar\'e duality, there is a closed $2d-i$ form
  $\beta$ with compact support, such that
   $\int_N\alpha\wedge \beta \not=0$. Let $U\subset N$ be the
  largest open set over which $f$ is a submersion. 
Then $f|_{f^{-1}U}$  is a $C^\infty$ fibre bundle. One has that $ V=\int_{f^{-1}(y)}
  \omega^e$ is independent of $y\in U$, because the fibres are
  homologous. Furthermore, $V\not=0$,  since it
  is  $e!$ times the volume of a fibre with respect to the
  induced K\"ahler metric \cite[p 31]{gh}. By Fubini's theorem, and the fact that
  $M-f^{-1}U$ has measure zero
$$\int_M f^*\alpha\wedge (f^*\beta \wedge \omega^{e}) = \int_{f^{-1}U} f^*(\alpha\wedge \beta)\wedge
\omega^{e} =V\int_N\alpha\wedge \beta \not=0$$
Therefore $f^*\alpha$ defines a nonzero cohomology class.
\end{proof}

\begin{rmk}
  The result is false without the K\"ahler assumption. The Hopf manifold $M=(\C^2-0)/2^\Z$
maps holomorphically onto $\PP^1$, but the map $H^2(\PP^1) \to H^2(M)$ is
not injective because $H^2(M)=0$.
\end{rmk}

\begin{prop}
 If $g\ge
2$, then $P_n(g)$ is not K\"ahler when $n\ge 2$.
\end{prop}

\begin{proof}
Let us assume that there is a compact connected K\"ahler manifold $M$
with $\pi_1(M)\cong P_n(g)$. We shall eventually produce
a contradiction.  We have a homorphism $J_n:\pi_1(M)\to \Pi_g^n$,
which we can decompose as a product of $n$ homomorphisms
$h_i:\pi_1(M)\to \Pi_g$. We note that $\ker h_i$ is finitely
generated by \eqref{eq:FN}. A theorem of Catanese \cite[thm 4.3]{catanese}, which refines
an earlier theorem of Beauville and Siu \cite[thm
2.11]{abc}, shows that there exists genus $g$ compact Riemann surfaces $Y_i$
and surjective holomorphic maps $f_i:M\to C_i$, with connected fibres, such that the induced maps
$\pi_1(M)\to \pi_1(Y_i)$ can be identified with $h_i$. Let 
$$f=
f_1\times\ldots\times  f_n:M\to Y_1\times \ldots \times Y_n=Y$$
and
$$f_{ij} = f_i\times f_j: M\to Y_i\times Y_j$$
We observe that $Y$ is a $K(\Pi_g^n,1)$ space, but $M$ is probably not
a $K(P_n(g),1)$. Nevertheless, by standard techniques \cite[thm
4.71]{hatcher}, we form a diagram of topological spaces
$$
\xymatrix{
 M\ar[r]^{f}\ar[d] & Y \\ 
 K(P_n(g),1)\ar[ru] & 
}
$$
which commutes up to homotopy. This means that we can factor the map $f^*$
on cohomology through $H^*(P_n(g))$. Similar remarks apply to
$f_{ij}$. Combining this observation with results of the previous
section allows us to draw several conclusions. By theorem~\ref{thm:abPn} (A),
\begin{equation}
  \label{eq:fstar}
  f^*:H^1(Y)\to H^1(M)
\end{equation}
 is an isomorphism, because Hurewicz gives  an
isomorphism $H^1(M)\cong H^1(P_n(g))$. Theorem~\ref{thm:abPn} (C) implies
that
\begin{equation}
  \label{eq:fijstar}
  f_{12}^*:H^4(Y_{12})\to H^4(M)
\end{equation}
is zero.
As a consequence, $f_{12}$ cannot be surjective by lemma \ref{lemma:inj}. 
So $\dim f_{12}(M)$ is either $0$ or $1$. The first possibility can be ruled out, because
 $f_1$ is
surjective and it factors through $f_{12}$. Therefore  $f_{12}(M)$ is a possibly
singular  compact complex
curve. Let $C$ be the normalization of $f_{12}(M)$, and let $g'$ be its genus. Then $f_{12}$
factors through $C$. Therefore, both $f_1$ and $f_2$ factors
through $C$. Thus $C$ is a branched cover of both $Y_1$ and $Y_2$.
This implies the genus $g'\ge g$. On the other hand, since $\pi_1(M)$
must surject onto $\pi_1(C)$,
   corollary~\ref{cor:SigmaPg}  implies $g'\le g$. So $g=g'$.  It follows that the maps $C\to Y_1$ and $C\to Y_2$
   are both isomorphisms.  This forces $\dim f_{12}^*H^1(Y_{12})=2g$.
We now have contradiction, because
$$\dim f_{12}^*H^1(Y_{12})+\sum_{i=3}^n \dim f_i^*H^1(Y_i)$$
is at most $2(n-1)g$, but it should be $2ng$ by \eqref{eq:fstar}

\end{proof}

When $M$ is compact K\"ahler, 
Green and Lazarsfeld \cite{gl} showed that positive dimensional
 components of certain cohomology jump loci are translates of subtori. Beauville
 \cite{beauville} gave  more precise information in the case of $\Sigma^1(\pi_1(M))$.

\begin{thm}[Beauville]\label{thm:beauville}
  Let $M$ be a  compact K\"ahler manifold. There is no
  untranslated  torus component of $\Sigma^1(\pi_1(M))$  of dimension
  $2$ or of odd dimension.
  A torus component of $\Sigma^1(\pi_1(M))$  of dimension
   $2g\ge 4$ is given by
  $f^*Char(C)$ for some holomorphic map $f:X\to C$, with connected
  fibres, onto a Riemann surface with genus $g$.
\end{thm}

\begin{proof}
This follows immediately from \cite[chap V, prop 1.7]
{arapura2}. However, it was already implicitly contained Beauville's
paper \cite{beauville}. Let us explain how to deduce it from results
proved there, since it is a bit more direct than going through \cite{arapura2}.
By Hodge theory, we have an isomorphism of Lie groups
$$Hom(\pi_1(M), U(1))\cong Pic^\tau(M)$$
where the group on the right is the group of line bundles with torsion
first Chern class.
Let 
$$\pm S^1(M)=\{L\in Pic^\tau(M)\mid H^1(M, L^{\pm 1})\not=0\}$$
By \cite[prop 3.5]{beauville},  we can identify
$$\Sigma^1(X)\cap Hom(\pi_1(M), U(1)) = S^1(M)\cup -S^1(M)$$
An algebraic torus in $\Sigma^1(M)$ of complex dimension $d$ maps to a torus on the right of real dimension $d$.
Such a torus would have to be of the form $f^*Pic^0(C)$ for a surface of genus $g=d/2>1$ by  \cite[thm 2.2]{beauville}.
Since the complex Zariski closure of $f^*Pic^0(C)$ can be identified with $f^*Char(C)$, we obtain
statement given in the theorem.

\end{proof}

The following lemma completes the proof of theorem~\ref{thm:hypRS}.

\begin{lemma}
The group $P_n(1)$ is not K\"ahler.  
\end{lemma}

\begin{proof}
  Suppose $M$ is a compact K\"ahler manifold with $\pi_1(M)\cong
  P_n(1)$. Then, by corollary~\ref{cor:SigmaP1}, $\Sigma^1(\pi_1(M))$
  contains a torus of dimension $2n-2$.  When $n=2$, this contradicts 
  theorem \ref{thm:beauville}. When $n\ge 3$, the theorem
  implies that we have a surjective homomorphism $P_n(1)\to \Pi_g$,
  with $g=n-1$, but this gives a contradiction to the last part of corollary~\ref{cor:SigmaP1}.
\end{proof}

\begin{lemma}
 The group $P_n(\C^*)$ is not K\"ahler.  
\end{lemma}

\begin{proof}
Let $X=\C^*$.
  Suppose $M$ is a compact K\"ahler manifold with $\pi_1(M)\cong
  P_n(X)$. The first Betti number of $M$ is even by Hodge theory.
  However proposition \ref{prop:Cstar} implies that $\rank H_1(P_2(X))=3$.
  So we must have $n\ge 3$. By corollary \ref{cor:Cstar} $\Sigma^1(\pi_1(M))$
  must contain an $(n-1)$ dimensional torus $T_{12}$. Let $T\subset \Sigma^1(\pi_1(M))$
be the largest torus containing $T_{12}$. We must have  $\dim T\ge 4$ by theorem \ref{thm:beauville}.
The same theorem would imply that $T= f^*Char(C)$ for some surjection $f:M\to C$ onto a curve
of genus $g\ge 2$. Let $h:P_n(X)\to \Pi_g$ be the corresponding
surjective homomorphism. We would
have $T_{12}\subseteq h^*Char(\Pi_g)$. However, this  contradicts corollary \ref{cor:Cstar}.
\end{proof}

\section{Braid groups of   higher dimensional projective manifolds}

A basic fact observed by Birman \cite[thm 1.5]{birman} is that pure
braid groups become uninteresting above two real dimensions.

\begin{thm}[Birman]\label{thm:birman}
  If $X$ is a $C^\infty$ manifold of  (real) dimension $d\ge 3$, then
  $J_n$ induces an isomorphism
$$P_n(X)\cong \pi_1(X)^n$$
\end{thm}

This can be proved by applying the following lemma    repeatedly  to each
$\Delta_{ij}\subset X^n$.

\begin{lemma}\label{lemma:purity}
  If $Z\subset X$ is a closed submanifold of codimension $3$ or more, $\pi_1(X-Z)\cong\pi_1(X)$.
\end{lemma}

\begin{proof}
Let $d$ be the codimension of $Z$.
  Let $T\subset X$ be a closed tubular neighbourhood of $Z$, and let
  $T^o=T-\partial T$ be the open neighbourhood. Then $\partial T$ is an
  $S^{d-1}$-bundle over $Z$. Therefore $\pi_1(\partial T)\cong
  \pi_1(Z)\cong \pi_1(T)$. Therefore, by Van Kampen,
$$\pi_1(X) \cong
  \pi_1(X-T^o)*_{\pi_1(\partial T)} \pi_1(T) \cong \pi_1(X-T^o)\cong \pi_1(X-Z)$$
\end{proof}

The full braid group is a bit more interesting,
especially in the context of projective groups, where there are very
few known constructions.
We recall that  given two groups $H$, $G$, and a $G$-set $I$,
the wreath product  $H\wr_I G$ is  the semidirect product $H^I\rtimes
G$. When $G=S_n$, we can take $I=\{1,2\ldots, n\}$ with the standard $G$-action.
When we write $H\wr S_n$, without specifying $I$, this is what we mean.
Clearly a general wreath product $H\wr_I G$, where $G$ is finite  and  $I$ is a finite and faithful $G$-set, can be embedded as a subgroup of
  finite index in a standard wreath product $H\wr S_n$.
We have the following corollary to theorem \ref{thm:birman}.

\begin{cor}\label{cor:birman}
 If $X$ is $C^\infty$ manifold of  (real) dimension $d\ge 3$, then
 $$B_n(X)\cong \pi_1(X)\wr S_n$$
\end{cor}

\begin{proof}
Covering space theory immediately shows that $B_n(X)$ is an extension
of $S_n$ by $\pi_1(X)^n$, but it is not clear it splits. Instead we
work with a suitable fibration.
  Let $R=\R^4$. Then $B_n(R)= S_n$ by theorem \ref{thm:birman}.
Let $p_1:(X\times R)^n=X^n\times R^n\to X^n$ and $p_2:(X\times R)^n\to R^n$  denote
the projection onto the first and second factors. 

 Let $\tilde U= (X\times
R)^n-p_2^{-1}\Delta_n(R)$, $\tilde V= (X\times R)^n-
p_1^{-1}\Delta_n(X) $, $U = \tilde U/S_n$, and
$V=\tilde V/S_n$. The projection
$p_2:\tilde U\to C_n(R)$ induces  a fibration $U\to SC_n(R)$ with
fibre $X^n$. Since $\pi_2(SC_n(R))=\pi_2(C_n(R))=0$ by  lemma~\ref{lemma:FN}, we have an exact sequence
$$1\to \pi_1(X)^n\to \pi_1(U)\to \pi_1(SC_n(R))\to 1$$
Moreover, the section $SC_n(R)\to U$ defined by $u\mapsto (u,0)\mod S_n$, splits the
sequence. By the theorem $\pi_1(SC_n(R))=S_n$. It remains to identify
$\pi_1(U)$ with $B_n(X)$. We have compatible inclusions $\tilde U
\subset C_n(X\times R)$ and $U\subset SC_n(X\times R)$.
This gives rise to a diagram
$$
\xymatrix{
  1\ar[r] & \pi_1(\tilde U)\ar[r]\ar[d]^{\cong} & \pi_1(U)\ar[r]\ar[d]& S_n\ar[r]\ar[d]^{=} & 1 \\ 
 1\ar[r] & P_n(X\times R)\ar[r] & B_n(X\times R)\ar[r] & S_n\ar[r] & 1
}
$$
where the first vertical map is an isomorphism by the previous
lemma. Therefore the middle vertical map is an isomorphism.
A similar argument shows that $V\to SC_n(X)$ is a fibration with fibre
$R^n$, and $B_n(X\times R)\cong \pi_1(V)$. Putting these facts
together gives an isomorphism $\pi_1(U)\cong B_n(X)$.

\end{proof}

\begin{thm}\label{thm:3}
 If $X$ is a  projective manifold of (complex) dimension $d\ge 2$. Then
 \begin{equation}\label{eq:BmWr}
 B_n(X) \cong \pi_1(X)\wr S_n,
\end{equation}
and this group is projective
\end{thm}

\begin{proof}
By corollary \ref{cor:birman}, $B_n(X)$ is a wreath product.
  The only thing to check is that $B_n(X)$ is projective.
After replacing $X$ by $X\times \PP^1$, we can assume $d\ge 3$. Note
that replacing $X$ by $ X\times \PP^1$ will not effect
$\pi_1(X)$, and consequently not $B_n(X)$. The variety $Y=S^nX$ is
projective, and the image $D$ of $\Delta(X)$ has codimension at least
$3$. Choose an embedding $Y\subset \PP^N$. Let $S\subset Y$ be a
surface obtained intersecting $Y$ with $\dim Y-2$ hyperplanes in
general position. Observe that $S\subset Y-D= SC_n(X)$, so it is
smooth by Bertini.
By the Lefschetz theorem of Hamm-Le \cite[thm 1.1.3]{hl}, or
Goresky-Macpherson \cite[p 153]{gm},
 $\pi_1(S)\cong \pi_1(Y-D)\cong B_n(X)$. (As an aside, we remark that the
 statement in \cite{gm} is more general, but harder to use ``out of the box''.)
\end{proof}

\begin{cor}
  Suppose that $H$ is projective,  $G$ is a  finite group, and $I$ is a finite faithful $G$-set.
The wreath product $H\wr_I G$ is projective.
\end{cor}

\begin{proof}
 Suppose that $H=\pi_1(X)$, with $X$ a smooth projective manifold. After replacing $X$ by $X\times \PP^2$,
 we can assume $\dim X\ge 2$. The theorem implies that $ H \wr S_n$ is
 projective.  The result for more general
 wreath products follows from this, the above remarks, and the fact
 that a finite index subgroup of a projective group is also  projective.
\end{proof}

\begin{rmk}\label{rmk:Serre}
  When $H$ is trivial, this recovers an old result of Serre \cite{serre} that a
  finite group is projective. Actually, the result is not explicitly stated, but it is a (known) corollary of \cite[prop
  15]{serre} and the weak Lefschetz theorem.
 Later on, Shafarevich \cite[chap IX, \S 4.2]{shafarevich},
 perhaps unaware of Serre's implicit result, gave a direct and elementary proof of this.
 Our argument is closely related to the one used by Shafarevich.
\end{rmk}

\section{Characteristic $p$}

Fix an algebraically closed field $k$. Given a
variety $X$ defined over $k$, Grothendieck \cite{sga} defined its \'etale
fundamental group $\pi^{et}_1(X)$, which is a profinite group. When $k=\C$, $\pi^{et}_1(X)$
 is the profinite completion of the usual fundamental group.
Therefore, when $X$ is smooth and projective, all the standard results
about K\"ahler groups, including the ones proved so far in this paper,
can be carried over.

Now suppose that $k$ has characteristic $p>0$. The first
thing to observe is that the $p$-part of the fundamental group can be a bit pathological.
It is better,  for our purposes, to consider the maximal pro-(prime to
$p$)  quotient $\pi_1^{p'}(X)$. 
Let $\cP(p)$ denote the class of pro-finite groups that can arise as
$\pi_1^{p'}(X)$, for a smooth projective variety defined over $k$. (As
the notation suggests, this does not depend on $k$ \cite[prop
1.1]{arapura-p}.)
 The author \cite[thm 3.2]{arapura-p} has found an analogue of  the
 result of \cite{abr} used above. Therefore the arguments used in
 lemmas \ref{lemma:artin} and \ref{lemma:P1}  can be easily modified to show that:

 \begin{prop} Suppose that $p>2$.
   If $G= P_n(\RR^2)$ with $n\ge 2$, or if $G= P_n(0)$ with $n\ge
   4$, then the pro-(prime to $p$) completion of $G$ does not lie in $\cP(p)$. 
 \end{prop}

 We end with the following question.

\begin{quest}
 If  $G=P_n(g), g>0, n\ge 2$, does the pro-(prime to $p$) completion of $G$  lie in $\cP(p)$? 
\end{quest}

Our proof of theorem~\ref{thm:hypRS} for these cases used the theorem
Beauville-Catanese-Siu and theorem \ref{thm:beauville}. These
results rely on consequences of
Hodge theory that can fail in positive characteristic, so it is not clear how to adapt these arguments.

\end{document}